\documentclass[preprint,review,10pt]{amsart}
\usepackage{amsmath,amssymb,amsfonts,xspace}
\newtheorem{theorem}{Theorem}[section]
\newtheorem{lemma}[theorem]{Lemma}

\theoremstyle{definition}
\newtheorem{definition}[theorem]{Definition}
\newtheorem{example}[theorem]{Example}

\newtheorem{corollary}[theorem]{Corollary}
\newtheorem{lem}[theorem]{Lemma}

\theoremstyle{remark}
\newtheorem{remark}[theorem]{Remark}

\numberwithin{equation}{section}

\begin{document}

\title{On $\phi$-$(k,n)$-absorbing $\delta$-primary hyperideals
  }

\author{M. Anbarloei}
\address{Department of Mathematics, Faculty of Sciences,
Imam Khomeini International University, Qazvin, Iran.
}

\email{m.anbarloei@sci.ikiu.ac.ir }





\begin{abstract}
In this paper, we aim to present the notion of  $\phi$-$(k,n)$-absorbing $\delta$-primary hyperideals in a Krasner $(m,n)$-hyperring $G$. 
\end{abstract}
\maketitle
\section{Introduction}
The concepts of  $\phi$-2-absorbing primary and 2-absorbing $\delta$-primary ideals in a commutative ring as two extensions of 2-absorbing primary ideals were proposed in \cite{bada} and \cite{bmb3}, respectively. Afterwards, Yavuz et al. presented the notion of 2-absorbing $\phi$-$\delta$-primary ideals  which merges these two  concepts \cite{Yavuz}. Let $\mathcal{I}(R)$ be  the set of all ideals of   a commutative ring $R$. Suppose that $\delta$ is an expansion  function from $ \mathcal{I}(R)$ to $\mathcal{I}(R)$ and $\phi$ is a reduction  function from $ \mathcal{I}(R)$ to $\mathcal{I}(R) \cup \{\varnothing\}$. A  proper ideal $I$ 
of $R$ is called a 2-absorbing $\phi$-$\delta$-primary ideal if   $x_1x_2x_3 \in I-\phi(I)$ for $x_1,x_2,x_3 \in R$, then  $x_1x_2 \in I$ or $x_1x_3 \in \delta(I)$ or $x_2x_3 \in \delta(I)$.

Hyperstructures, concretely hypergroups, were innovated  by
Marty in 1934 \cite{s1}.  Some review of the hyperstructures  can be found in  \cite{amer2, s2, s3,  davvaz1, s6, s5,s1, davvaz2, s4, s10, jian}. An algebraic hyperstructure is equipped with at least one  hyperoperation, that is, a multi-valued operation. A class of
algebraic hyperstructures called Krasner $(m,n)$-hyperring was introduced by Mirvakili and Davvaz in \cite{d1}. This notion is a subclass of $(m,n)$-hyperring proposed in \cite{cons} and an extension of Krasner hyperring presented in \cite{Krasner}. 
 Two  of the most important
hyperideals  in a  Krasner $(m, n)$-hyperring, namely   $n$-ary prime and $n$-ary primary hyperideals  were presented   in \cite{sorc1}.  Alghough, these two concepts are different in many
aspects, they have some similar properties. Then, the notion of $\delta$-primary  hyperideals  which covers the concepts of the prime and primary hyperideals  was investigated in \cite{mah3}. The concept of $(k,n)$-absorbing (primary) hyperideals  was defined by Hila et al. \cite{rev2}. 
Afterwards, weakly $(k,n)$-absorbing  ( primary ) hyperideals were introduced by Davvaz et al.  in \cite{weak}. Moreover, the notion of $\phi$-$(k,n)$-absorbing (primary) hyperideals of  Krasner $(m,n)$-hyperrings was studied in \cite{mah6}. 

Therefore, it is  natural to examine
whether it is possible to get a unified approach to investigating these structures. 
In this paper, we define the notion of $\phi$-$(k,n)$-absorbing $\delta$-primary hyperideals in a commutative Krasner $(m,n)$-hyperring $G$ which covers all the  above mentioned concepts. 
Among the main results of this paper,  we show that  $\phi$-$(k,n)$-absorbing hyperideals and $\phi$-$(k,n)$-absorbing $\delta$-primary hyperideals are different concepts in Example \ref{exa2}. Although every  $(k,n)$-absorbing $\delta$-primary hyperideal of $G$ is  a $\phi$-$(k,n)$-absorbing $\delta$-primary hyperideal, Example \ref{e22} verifies that the converse is not true in general. We conclude that 
that the radical of the $\phi$-$(k,n)$-absorbing $\delta$-primary hyperideal $I$ of $G$ is $\phi$-$(k,n)$-absorbing $\delta$-primary when ${\bf r}(\phi(I)) \subseteq \phi({\bf r}(I))$ and ${\bf r}(\delta(I)) \subseteq \delta({\bf r}(I))$ in Theorem \ref{1}. We show that thses two condictions are crucial in Example \ref{ignor}. We give some conditions under which the intersection of a family of $\phi$-$(k,n)$-absorbing $\delta$-primary hyperideals is a $\phi$-$(k,n)$-absorbing $\delta$-primary hyperideal in Theorem \ref{kharazi}. We obtain that if $I$ is a strongly $\phi$-$(k,n)$-absorbing $\delta$-primary hyperideal of $G$  that is not  $(k,n)$-absorbing $\delta$-primary, then $g$-product of $kn-k+1$ copies of  $I$ is a subset of $\phi(I)$ in Corollary \ref{weak5} (i). Example \ref{mansooreh} verifies that a proper hyperideal $I$ of $G$ need not to be  $\phi$-$(k,n)$-absorbing $\delta$-primary when  $g(I^{(kn-k+1)}) \subseteq \phi(I)$. Besides, we analyze the images and inverse images of $\phi$-$(k,n)$-absorbing $\delta$-primary hyperideals
under homomorphism. Finally, we provide some characterizations of this notion on cartesian product of Krasner $(m, n)$-hyperrings.



\section{Krasner $(m,n)$-hyperrings}
In this section, we summarize some basic definitions  related to Krasner $(m,n)$-hyperrings. \\
Assume that $G \neq \varnothing$  and $P^*(G)$ is the
set of all the non-empty subsets of $G$.  An $n$-ary hyperoperation on $G$ is a map  $f : G^n \longrightarrow P^*(G)$
 and the couple  $(G, f)$ is  an $n$-ary hypergroupoid. The abbreviated  notation $t_i^j$   denote  the sequence $t_i, t_{i+1},\ldots, t_j$ if $j \geq i$ and it is the empty symbol if $j < i$. Moreover, if $s_{i+1} =\cdots = s_j = s$, we write $f(t^i_1, s^j_{i+1}, r^n_{j+1})=f(t^i_1, s^{(j-i)}, r^n_{j+1})$.
 We define
$f(G^n_1) = f(G_1,\ldots, G_n) = \bigcup \{f(t^n_1) \ \vert \  t_i \in  G_i, 1 \leq i \leq  n \}$ where $G_1^n$   are non-empty subsets of $G$. 
Let $f$ be an $n$-ary hyperoperation and $v = l(n- 1) + 1$. Then, $v$-ary hyperoperation $f_{(l)}$ is defined by
$f_{(l)}(t_1^{l(n-1)+1}) = f\bigg(f(\ldots, f(f(t^n _1), t_{n+1}^{2n -1}),\ldots), t_{(l-1)(n-1)+1}^{l(n-1)+1}\bigg).$ Let $\mathbb{S}_n$ be the group of all permutations of $\{1, \ldots, n\}$. If $f(t_1^n) = f(t_{\sigma(1)}^{\sigma(n)})$ for all $ \sigma \in \mathbb{S}_n$ and for all $t_1^n \in G$, then $(G,f)$ is commutative.

\begin{definition}
\cite{d1} An algebraic hyperstructure $(G, f, g)$, or simply $G$, is   a Krasner $(m, n)$-hyperring if  
\begin{itemize}
\item[\rm(1)]~ $(G, f)$ is a canonical $m$-ary hypergroup;
\item[\rm(2)]~ $(G, g)$ is an $n$-ary semigroup having $0$ as $g(t_1^{i-1},0,t_{i+1}^n) = 0$, for each $t^n_ 1 \in G$;
\item[\rm(3)]~ $g(t^{i-1}_1, f(s^m _1 ), t^n _{i+1}) = f(g(t^{i-1}_1, s_1, t^n_{ i+1}),\ldots, g(t^{i-1}_1, s_m, t^n_{ i+1}))$ for all $i \in \{1,\ldots,n\}$ and $t^{i-1}_1 , t^n_{ i+1}, s^m_ 1 \in G$. 
\end{itemize}
\end{definition}
Throughout this paper,  $G$ is a commutative Krasner $(m,n)$-hyperring with the scalar identity $1$.

Let $\varnothing \neq H \subseteq G$. Then, we say that $H$ is a subhyperring of $G$ if $(H, f, g)$ is a Krasner $(m, n)$-hyperring. Let $\varnothing \neq I \subseteq G$.  $I$  is a hyperideal  of $G$ if $(I, f)$ is an $m$-ary subhypergroup
of $(G, f)$ and $g(t^{i-1}_1, I, t_{i+1}^n) \subseteq I$, for each $t^n _1 \in  G$ and  $1 \leq i \leq n$. By  $\mathcal{L}(G)$, we mean all hyperideals of $G$.

\begin{definition}
 \cite{sorc1} A proper hyperideal $I$ of  $G$ is    prime  if for hyperideals $I_1^n$ of $G$, $g(I_1^ n) \subseteq I$ implies $I_i \subseteq I$ for some $1 \leq i \leq n$.
\end{definition}
 
By Lemma 4.5 in \cite{sorc1}, it is proved that  a proper hyperideal $I$ of $G$ is  prime  if for each $t^n_ 1 \in G$, $g(t^n_ 1) \in I$ implies  $t_i \in I$ for some $1 \leq i \leq n$. 

\begin{definition} \cite{sorc1}   Let $I$ be a hyperideal of $G$. ${\bf r}(I)$, the radical of $I$,  is the intersection of  all  prime hyperideals of $G$ containing $I$.  We get ${\bf r}(I)=G$ if the set of all prime hyperideals   containing $I$ is empty.
\end{definition}
 If $t \in {\bf r^{(m,n)}}(I)$, then 
 there exists $v \in \mathbb {N}$ with $g(t^ {(v)} , 1_G^{(n-v)} ) \in I$ for $v \leq n$, or $g_{(l)} (t^ {(v)} ) \in I$ for $v = l(n-1) + 1$ (See, Theorem 4.23 in \cite{sorc1}).
\begin{definition}
\cite{sorc1} A proper hyperideal  $I$ of $G$ is     primary  if $g(t^n _1) \in I$ and $t_i \notin I$ implies $g(t_1^{i-1}, 1,t _{ i+1}^n) \in {\bf r}(I)$ for some $1 \leq i \leq n$.
\end{definition}
Let $I$ be a primary hyperideal of $G$. Theorem 4.28 in \cite{sorc1} proves that the radical of  $I$ is prime.
 \begin{definition} 
\cite{rev2} Let $I$ be a proper hyperideal of  $G$. We say that $I$ is an 
\begin{itemize}
\item[\rm(1)]~ $(k,n)$-absorbing hyperideal  if for  $t_1^{kn-k+1} \in G$, $g(t_1^{kn-k+1}) \in I$ implies that there exist $(k-1)n-k+2$ of the $t_i^,$s whose $g$-product is in $I$. 
\item[\rm(2)]~ $(k,n)$-absorbing primary hyperideal  if for  $t_1^{kn-k+1} \in G$, $g(t_1^{kn-k+1}) \in I$ implies that $g(t_1^{(k-1)n-k+2}) \in I$ or a $g$-product of $(k-1)n-k+2$ of the $t_i^,$s, except $g(t_1^{(k-1)n-k+2})$, is in ${\bf r}(I)$.
\end{itemize}
\end{definition}


\section{$\phi$-$(k,n)$-absorbing $\delta$-primary hyperideals}
A function $\phi$ from $\mathcal{L}(G)$ to $\mathcal{L}(G) \cup \{\varnothing\}$ refers to a reduction function if  $\phi(I_1) \subseteq I_1$ and $I_1 \subseteq I_2$ implies $\phi(I_1) \subseteq \phi(I_2)$ for each $I_1,I_2 \in \mathcal{L}(G)$. For example, the functions $\phi_{\varnothing}, \phi_0, \phi_1,  \phi_n: \mathcal{L}(G) \longrightarrow \mathcal{L}(G) \cup \{\varnothing\}$, defined by $\phi_{\varnothing}(I)=\varnothing, \phi_0(I)=0,\phi_1(I)=I$ and $\phi_n(I)=g(I^{(n)})$ for $n \geq 2$, are reduction functions. Moreover, a function $\delta$ from $\mathcal{L}(G)$ to $\mathcal{L}(G)$ is called  an expansion  function if  $I_1  \subseteq \delta(I_1)$ and $I_1 \subseteq I_2$ implies $\delta(I_1) \subseteq \delta(I_2)$ for each $I_1,I_2 \in \mathcal{L}(G)$. For instance, the functions $\delta_0, \delta_1, \delta_G,\delta_{\mathcal{L}(G)}: \mathcal{L}(G) \longrightarrow \mathcal{L}(G)$, defined by $\delta_0(I)=I, \delta_1(I)={\bf r}(I),\delta_G(I)=G$ and $\delta_{\mathcal{L}(G) }(I)=\cap_{I \subseteq J \in \mathcal{L}(G) }J$, are expansion  functions. In this section, we propose a unified framework for existing structures such as $\phi$-$(k,n)$-absorbing  hyperideals and $(k,n)$-absorbing $\delta$-primary hyperideals in a Krasner $(m,n)$-hyperring.   Throughout this section,  $\phi$ and $\delta$ are reduction and expansion fuctions, respectively.
\begin{definition}
Let $I$  be a proper hyperideal of $G$ and $k \in \mathbb{Z}^+$. $I$ refers to a  $\phi$-$(k,n)$-absorbing $\delta$-primary hyperideal provided that for $t_1^{kn-k+1} \in G$,  $g(t_1^{kn-k+1}) \in I-\phi(I)$ implies  that  $g(t_1^{(k-1)n-k+2}) \in I$ or a $g$-product  of $(k-1)n-k+2$ of  $t_i^,$s, except $g(t_1^{(k-1)n-k+2}),$ is in $\delta(I)$.
\end{definition} 
\begin{example} \label{hanieh} 
The construction $G=\mathbb{Z}_{12}/\mathbb{Z}_{12}^*$ is a Krasner hyperring where $\mathbb{Z}_{12}^*=\{1,5,7,11\}$ is a multiplicative
group of units of $\mathbb{Z}_{12}$ by Example 3.2 in \cite{rev2}. In the hyperring,   $I=\{0\mathbb{Z}_{12}^*,2\mathbb{Z}_{12}^*,4\mathbb{Z}_{12}^*, 6\mathbb{Z}_{12}^*\}$ is a $\phi_0$-$(2,2)$-absorbing $\delta_1$-primary hyperideal of $G$. 
\end{example}
Recall from \cite{mah6} that a proper hyperideal $I$ of $G$ is called a $\phi$-$(k,n)$-absorbing hyperideal  if for  $t_1^{kn-k+1} \in G$, $g(t_1^{kn-k+1}) \in I-\phi(I)$ implies that there exist $(k-1)n-k+2$ of the $t_i^,$s whose $g$-product is in $I$. 
\begin{remark} \label{e1}
Every $\phi$-$(k,n)$-absorbing hyperideal of a Krasner $(m,n)$-hyperring is  a $\phi$-$(k,n)$-absorbing $\delta$-primary hyperideal.
\end{remark} 
The converse of Remark \ref{e1} need not to be hold  as it is shown in the following example. 
\begin{example} \label{exa2}
Consider the Krasner $(2,2)$-hyperring  $(G=[0,1],\oplus,\cdot)$ where the  $2$-ary hyperoperation  $``\oplus"$ defined by $x \oplus y=\{max\{x, y\}\}$ if $x \neq y$ or $[0,x]$ if otherwise and  $``\cdot"$
  is the usual multiplication on real numbers. Suppose  that   $\phi$ is  a function from $\mathcal{L}(G)$ to $\mathcal{L}(G) \cup \{\varnothing\}$ defined $\phi(Q)=\cap_{i=1}^{\infty} g(Q^{(i)})$ for  $Q \in \mathcal{L}(G)$. In this hyperring,  the hyperideal $I=[0,0.5]$ is a  $\phi$-$(2,2)$-absorbing $\delta_1$-primary hyperideal but it is not $\phi$-$(2,2)$-absorbing since $ 0.7 \cdot 0.8^2 \in I-\phi(I)$ and $ 0.7 \cdot 0.8, 0.8^2 \cdot 0.7 \notin I$.
\end{example}
 A proper hyperideal $I$ of $G$ refers to a $(k,n)$-absorbing $\delta$-primary hyperideal  if for  $t_1^{kn-k+1} \in G$, $g(t_1^{kn-k+1}) \in I$ implies $g(t_1^{(k-1)n-k+2}) \in I$ or a $g$-product of $(k-1)n-k+2$ of the $t_i^,$s, except $g(t_1^{(k-1)n-k+2})$, is in $\delta(I)$ \cite{mah7}.
\begin{remark} \label{e2}
Every $(k,n)$-absorbing $\delta$-primary hyperideal of a Krasner $(m,n)$-hyperring is  a $\phi$-$(k,n)$-absorbing $\delta$-primary hyperideal.
\end{remark} 
Note that the converse of Remark \ref{e2} may not be generally
true. 
\begin{example} \label{e22}
Let $G=\{0,1,2,3\}$. By Example 4.8 in  \cite{sorc1}, $G$ is a Krasner $(2,4)$-hyperring with the 2-ary hyperoperation  $+$ and 4-ary operation $g$ deﬁned by

\[\begin{tabular}{|c|c|c|c|c|} 
\hline $+$ & $0$ & $1$ & $2$ & $3$
\\ \hline  $0$ & $0$ & $1$ & $2$ & $3$
\\ \hline $1$ & $1$ & $\{0,1\}$ &  $3$ & $\{2,3\}$
\\ \hline $2$ & $2$ & $3$ & $0$ & $1$ 
\\ \hline $3$ & $3$ & $\{2,3\}$ &  $1$ & $\{0,1\}$
\\ \hline
\end{tabular}\]

and

\[ 
g(t_1^4)=
\begin{cases}
2, & \text{if $t_1^4 \in \{2,3\}$}\\
0, & \text{otherwise.}
\end{cases} \]
Although $I=\{0\}$ is a $\phi_0$-$(1,4)$-absorbing $\delta_0$-primary hyperideal of $G$, Example 2.2 in \cite{weak} verifies that it is not a $(1,4)$-absorbing $\delta_0$-primary hyperideal of $G$.
\end{example}
\begin{theorem}
If $I$ is a $\phi$-$(k,n)$-absorbing $\delta$-primary hyperideal of $G$, then $I$ is a $\phi$-$(u,n)$-absorbing $\delta$-primary hyperideal for all $u \geq k$.
\end{theorem}
\begin{proof}
It suffices to show  every $\phi$-$(k,n)$-absorbing $\delta$-primary hyperideal of $G$ is a $\phi$-$(k+1,n)$-absorbing $\delta$-primary hyperideal of $G$. Let $I$ be a $\phi$-$(k,n)$-absorbing $\delta$-primary hyperideal of $G$. Assume that $g(g(t_1^{n+2}),t_{n+3}^{(k+1)n-(k+1)+1}) \in I-\phi(I)$ for $t_1^{(k+1)n-(k+1)+1} \in G$. Put $g(t_1^{n+2})=s_1$. Then  $g(s_1,\ldots,t_{(k+1)n-(k+1)+1}) \in I$ or a $g$-product of $kn-k+1$ of the $r_i^,$s, except $g(s_1,\ldots, t_{(k+1)n-(k+1)+1})$, is in $\delta(I)$ as $I$ is a $\phi$-$(k,n)$-absorbing $\delta$-primary hyperideal of $G$. Since $\delta(I)$ is a hyperideal of $G$ and $t_1^{n+2} \in G$, we have $g(t_i,t_{n+3},\ldots,t_{(k+1)n-(k+1)+1}) \in \delta(I)$ for $i \in \{1,\ldots,n+2\}$
which implies  $I$ is $\phi$-$(k+1,n)$-absorbing $\delta$-primary hyperideal of $G$.
\end{proof}
The folowing  theorem provides conditions under which a $\phi$-$(k,n)$-absorbing $\delta$-primary hyperideal is $(k,n)$-absorbing $\delta$-primary.
\begin{theorem}
Let  $I$ be a proper hyperideal of $G$  such that $\phi(I)$ is a $(k,n)$-absorbing $\delta$-primay hyperideal of $G$. If $I$ is a $\phi$-$(k,n)$-absorbing $\delta$-primary hyperideal of $G$, then $I$ is a $(k,n)$-absorbing $\delta$-primary hyperideal of $G$.
\end{theorem}
\begin{proof}
Let  $g(t_1^{kn-k+1}) \in I$ for $t_1^{kn-k+1} \in G$ but $g(t_1^{(k-1)n-k+2}) \notin I$. Suppose that $g(t_1^{kn-k+1}) \in \phi(I)$. Since $\phi(I)$ is a $(k,n)$-absorbing $\delta$-primary hyperideal of $G$ and $g(t_1^{(k-1)n-k+2}) \notin \phi(I)$, we conclude that a $g$-product of $(k-1)n-k+2$ of the $t_i^,$s, except $g(t_1^{(k-1)n-k+2})$ is in $\delta(\phi(I))\subseteq \delta(I)$, as required. Now, assume that $g(t_1^{kn-k+1}) \notin \phi(I)$. Since $I$ is a $\phi$-$(k,n)$-absorbing $\delta$-primary hyperideal of $G$,  $g(t_1^{kn-k+1}) \in I - \phi(I)$ and $g(t_1^{(k-1)n-k+2}) \notin I$, we get the result that a $g$-product of $(k-1)n-k+2$ of the $t_i^,$s, except $g(t_1^{(k-1)n-k+2})$ is in $\delta(I)$.
\end{proof}
The next theorem provides us how to determine the radical of a $\phi$-$(k,n)$-absorbing $\delta$-primary hyperideal to be a $\phi$-$(k,n)$-absorbing $\delta$-primary hyperideal.

\begin{theorem} \label{1}
Let $I$ be a proper hyperideal of $G$. If $I$ is a $\phi$-$(k,n)$-absorbing $\delta$-primary hyperideal of $G$ such that  ${\bf r}(\phi(I)) \subseteq \phi({\bf r}(I))$ and  ${\bf r}(\delta(I)) \subseteq \delta({\bf r}(I))$, then ${\bf r}(I)$ is a $\phi$-$(k,n)$-absorbing $\delta$-primary hyperideal of $G$.
\end{theorem}
\begin{proof}
Let  $g(t_1^{kn-k+1}) \in {\bf r}(I)-\phi({\bf r}(I))$ for $t_1^{kn-k+1} \in G$ . Suppose that  all products of $(k-1)n-k+2$ of the $t_i^,$s except $g(t_1^{(k-1)n-k+2})$ are not in $\delta({\bf r}(I))$. Since $g(t_1^{kn-k+1}) \in {\bf r}(I)$, then there exists $v \in \mathbb{Z}^+$ with $g(g(t_1^{kn-k+1})^{(v)},1^{(n-v)}) \in I$, for $v \leq n$ or $g_{(l)}(g(t_1^{kn-k+1})^{(v)}) \in I$, for $v>n$, $v=l(n-1)+1$. In the first possibility, if $g(g(t_1^{kn-k+1})^{(v)},1^{(n-v)}) \in \phi(I)$, then we get $g(t_1^{kn-k+1}) \in {\bf r}(\phi(I)) \subseteq \phi({\bf r}(I))$ which is impossible. Therefore we have $g(g(t_1^{(v)},1_G^{(n-v)}),\ldots,g(t_{kn-k+1}^{(v)},1_G^{(n-v)})) \in I - \phi(I)$.   Then, we conclude that  $g(g(t_1^{(v)},1_G^{(n-v)}),\ldots,g(t_{(k-1)n-k+2}^{(v)}),1_G^{(n-v)}))=g(g(t_1^{(k-1)n-k+2)})^{(v)},1_G^{(n-v)}) \in I$ and so $g(t_1^{(k-1)n-k+2}) \in {\bf r}(I)$ as  $I$ is a $\phi$-$(k,n)$-absorbing $\delta$-primary hyperideal of $G$ and all products of $(k-1)n-k+2$ of the $t_i^,$s except $g(t_1^{(k-1)n-k+2})$ are not in ${\bf r}(\delta(I))$. For the other case, we are done by using a similar argument. Thus, ${\bf r}(I)$ is a $\phi$-$(k,n)$-absorbing  $\delta$-primary hyperideal of $G$.
\end{proof}
The following example shows that the conditions mentioned in Theorem \ref{1}, namely $``{\bf r}(\phi(I)) \subseteq \phi({\bf r} (I))"$ and $``{\bf r}(\delta(I)) \subseteq \delta({\bf r} (I))"$,   can not be ignored. 
\begin{example}\label{ignor}
Consider Krasner $(m,n)$-hyperring $G=H/I$ with ordinary addition and ordinary multiplication where $H=\mathbb{Z}_3[X,Y,Z]$ and $I=\langle X^3Y^3Z^3\rangle$.
 Then $I/I$ is a $\phi_0$-$(1,3)$-absorbing $\delta_0$-primary hyperideal of $G$. Note that ${\bf r}(\phi_0(I/I)) \nsubseteq \phi_0({\bf r} (I/I))$ and 
${\bf r} (I/I)$ is not a $\phi_0$-$(1,3)$-absorbing $\delta_0$-primary hyperideal of $G$ because $2XYZ+I=(2X+I)(Y+I)(Z+I) \in {\bf r} (I/I) - \phi_0({\bf r} (I/I))$ but  elements $(2X+I), (Y+I)$ and $(Z+I)$ are  not in  $\delta_0({\bf r} (I/I))$.
\end{example}
\begin{theorem}
Let  $I$ be  a $\phi$-$(1,n)$-absorbing $\delta$-primary hyperideal of $G$. Then $I$ is a $\phi$-$(2,n)$-absorbing $\delta$-primary hyperideal.
\end{theorem}
\begin{proof}
Assume that $I$ is a $\phi$-$(1,n)$-absorbing $\delta$-primary hyperideal of $G$. Let $t_1^{2n-1} \in G$ such that $g(g(t_1^n),t_{n+1},\ldots,t_{2n-1}) \in I-\phi(I)$. Therefore we get  $g(t_1^n) \in I$ or $g(1_G,t_{n+1}^{2n-1}) \in \delta(I)$.  This implies that  $g(t_i,t_{n+1},\ldots,t_{2n-1}) \in  \delta(I)$ for each $i \in \{1,\ldots,n\}$
 as $t_1^n \in G$ and $\delta(I)$ is a hyperideal of $G$. Hence $I$ is a $\phi$-$(2,n)$-absorbing $\delta$-primary hyperideal of $G$.
\end{proof}
\begin{theorem}
Let $\phi_1$ and $ \phi_2$  be    reduction functions such that for each $I \in  \mathcal{L}(G)$, $\phi_1(I) \subseteq \phi_2(I)$. If $I$ is a $\phi_1$-$(k,n)$-absorbing $\delta$-primary hyperideal of $G$, then $I$ is a $\phi_2$-$(k,n)$-absorbing $\delta$-primary hyperideal.
\end{theorem}
\begin{proof}
Assume that $g(t_1^{kn-k+1}) \in I-\phi_2(I)$ for $t_1^{kn-k+1} \in G$. Since $\phi_1(I) \subseteq \phi_2(I)$, we obtain $g(t_1^{kn-k+1}) \in I-\phi_1(I)$. Therefore, we get $g(t_1^{(k-1)n-k+2}) \in I$ or a $g$-product  of $(k-1)n-k+2$ of  $t_i^,$s, except $g(t_1^{(k-1)n-k+2}),$ is in $\delta(I)$ as  $I$ is a $\phi_1$-$(k,n)$-absorbing $\delta$-hyperideal of $G$. This shows that $I$ is a $\phi_2$-$(k,n)$-absorbing $\delta$-primary hyperideal.
\end{proof}
Here, we investigate
whether the union of the collection of $\phi$-$(k,n)$-absorbing $\delta$-primary hyperideals preserve the algebraic structure.
\begin{theorem}
The union of  a directed family of $\phi$-$(k,n)$-absorbing $\delta$-primary hyperideals of $G$ is a $\phi$-$(k,n)$-absorbing $\delta$-primary hyperideal.
\end{theorem}
\begin{proof}
Let $ \{I_u\}_{u \in \Delta}$ be a directed family of $\phi$-$(k,n)$-absorbing $\delta$-primary hyperideals of $G$. Assume that $g(t_1^{kn-k+1}) \in \bigcup_{u \in \Delta}I_u-\phi(\bigcup_{u \in \Delta}I_u)$ for $t_1^{kn-k+1} \in G$ such that none of $g$-product  of $(k-1)n-k+2$ of  $t_i^,$s is in $\delta(\bigcup_{u \in \Delta}I_u)$. Then, there exists a $v \in \Delta$ such that $g(t_1^{kn-k+1}) \in I_v-\phi(I_v)$ such  that none of $g$-product  of $(k-1)n-k+2$ of  $t_i^,$s is in $\delta(I_v)$. Therefore,  there exist $(k-1)n-k+2$ of  $t_i^,$s whose $g$-product is in $ I_v$ as $I_v$ is a  $\phi$-$(k,n)$-absorbing $\delta$-primary hyperideals of $G$. This implies that the $g$-product of  $(k-1)n-k+2$ of   $t_i^,$s is in $\bigcup_{u \in \Delta}I_u$. Consequently,  $\bigcup_{u \in \Delta} I_u$ is a $\phi$-$(k,n)$-absorbing $\delta$-primary hyperideal of $G$.
\end{proof}
The next theorem presents conditions under which the intersection of  a family of $\phi$-$(k,n)$-absorbing $\delta$-primary hyperideals is a $\phi$-$(k,n)$-absorbing $\delta$-primary hyperideal.
\begin{theorem} \label{kharazi}
Let $ \{I_u\}_{u \in \Delta}$ be a family of  hyperideals of $G$ such that $\delta(I_u)=\delta(I_v)$ and $\phi(I_u)=\phi(I_v)$ for all $u,v \in \Delta$. If $I_u$ is a $\phi$-$(k,n)$-absorbing $\delta$-primary hyperideal of $G$ for each $u \in \Delta$ and $\phi,\delta$ have the property of intersection
preserving, then $\bigcap_{u \in \Delta}I_u$ is a $\phi$-$(k,n)$-absorbing $\delta$-primary hyperideal of $G$.
\end{theorem}
\begin{proof}
Put $I=\bigcap_{u \in \Delta}I_u$. Assume that $g(t_1^{kn-k+1}) \in I-\phi(I)$ for $t_1^{kn-k+1} \in G$ such that none of $g$-product  of $(k-1)n-k+2$ of  $t_i^,$s is in $\delta(I)$. Hence, we obtain $g(t_1^{kn-k+1}) \in I_u-\phi(I_u)$ for each $u \in \Delta$ and none of $g$-product  of $(k-1)n-k+2$ of  $t_i^,$s is in $\delta(I_u)$. Since $I_u$ is a $\phi$-$(k,n)$-absorbing $\delta$-primary hyperideal of $G$ for each $u \in \Delta$, there exist $(k-1)n-k+2$ of  $t_i^,$s whose $g$-product is in $ I_u$  which implies the $g$-product of  $(k-1)n-k+2$ of   $t_i^,$s is in $I$. Thus, $I=\bigcap_{u \in \Delta}I_u$ is a $\phi$-$(k,n)$-absorbing $\delta$-primary hyperideal of $G$.
\end{proof}
Let $Q$ be a hyperideal of $(G,f,g)$. Then, $(G/Q,f,g)$ is the quotient Krasner $(m,n)$-hyperring of $G$ by $Q$ where $G/Q=\{f(t_1^{i-1},Q,t_{i+1}^m) \ \vert \ t_1^{i-1},t_{i+1}^m \in G$\} \cite{sorc1}.   Recall from \cite{weak} that a hyperideal $I$  of $G$ is called  a weakly $(k,n)$-absorbing primary hyperideal   if   $0 \neq g(t_1^{kn-k+1}) \in I$ for $t_1^{kn-k+1} \in G$ implies that  $g(t_1^{(k-1)n-k+2}) \in I$ or a $g$-product  of $(k-1)n-k+2$ of  $t_i^,$s ,except $g(t_1^{(k-1)n-k+2})$, is in ${\bf r}(I)$. Now, we say that $I$ is a weakly $(k,n)$-absorbing $\delta$-primary hyperideal of $G$ if   $0 \neq g(t_1^{kn-k+1}) \in I$ for $t_1^{kn-k+1} \in G$ implies that  $g(t_1^{(k-1)n-k+2}) \in I$ or a $g$-product  of $(k-1)n-k+2$ of  $t_i^,$s ,except $g(t_1^{(k-1)n-k+2})$, is in $\delta(I)$. Moreover, we say that $I$ is a strongly $\phi$-$(k,n)$-absorbing $\delta$-primary hyperideal of $G$ if $g(I_1^{kn-k+1}) \subseteq I-\phi(I)$ for  hyperideals $I_1^{kn-k+1}$ of $G$ implies that $g(I_1^{(k-1)n-k+2}) \subseteq I$ or a $g$-product of $(k-1)n-k+2$ of $I_i^,$s, except $g(I_1^{(k-1)n-k+2}),$ is a subset of $\delta(I)$. In Theorem \ref{weak} and Theorem \ref{madar}, we assume that all $\phi$-$(2,2)$-absorbing $\delta$-primary  hyperideals are strongly $\phi$-$(2,2)$-absorbing $\delta$-primary  hyperideals.
\begin{theorem} \label{weak}
Let $I$ be a proper hyperideal of $G$. Then $I$  is   a  $\phi$-$(2,2)$-absorbing $\delta$-primary hyperideal of $G$ if and only if   $I/\phi(I)$ is a  weakly $(2,2)$-absorbing $\delta_q$-primary hyperideal of $G/\phi(I)$. 
\end{theorem}
\begin{proof}
 $\Longrightarrow$ Let  $I$  be  a strongly $\phi$-$(2,2)$-absorbing $\delta$-primary hyperideal of $G$ and \\

$\phi(I) \neq g(f(t_{11}^{1(i-1)},\phi(I),t_{1(i+1)}^{1m}),f(t_{21}^{2(i-1)},\phi(I),t_{2(i+1)}^{2m}),f(t_{31}^{3(i-1)},\phi(I),t_{3(i+1)}^{3m}))$

$\hspace{0.7cm}=f(g(t_{11}^{31}),\ldots,g(t_{1(i-1)}^{3(i-1)}),\phi(I),g(t_{1(i+1)}^{3(i+1)}),\ldots,g(t_{1m}^{3m}))$

$\hspace{0.7cm} \in I/\phi(I)$. 

for $t_{11}^{1m},t_{21}^{2m},t_{31}^{3m}\in G$.
Then we conclude that

$\hspace{1cm} f(g(t_{11}^{31}),\ldots, g(t_{1(i-1)}^{3(i-1)}),0,g(t_{1(i+1)}^{3(i+1)}),\ldots,g(t_{1m}^{3m}))$  

$\hspace{1cm}=g(f(t_{11}^{1(i-1)},0,t_{1(i+1)}^{1m}),f(t_{21}^{2(i-1)},0,t_{2(i+1)}^{2m}),f(t_{31}^{3(i-1)},0,t_{3(i+1)}^{3m}))$

$\hspace{1cm}\subseteq  I-\phi(I).$\\
This implies that

 $\hspace{1cm}g(f(t_{11}^{1(i-1)},0,t_{1(i+1)}^{1m}),f(t_{21}^{2(i-1)},0,t_{2(i+1)}^{2m}))$

$ \hspace{1.5cm}=f(g(t_{11}^{21}),\ldots,g(t_{1(i-1)}^{2(i-1)}),0,g(t_{1(i+1)}^{2(i+1)}),\ldots,g(t_{1m}^{2m}))\subseteq I$\\ or 

$\hspace{1cm}g(f(t_{21}^{2(i-1)},0,t_{2(i+1)}^{2m}),f(t_{31}^{3(i-1)},0,t_{3(i+1)}^{3m}))$

 $\hspace{1.5cm}=f(g(t_{21}^{31}),\ldots,g(t_{2(i-1)}^{3(i-1)}),0,g(t_{2(i+1)}^{3(i+1)}),\ldots,g(t_{2m}^{3m}))\subseteq \delta(I)$\\
or

$\hspace{1cm}g(f(t_{11}^{1(i-1)},0,t_{1(i+1)}^{1m}),f(t_{31}^{3(i-1)},0,t_{3(i+1)}^{3m}))$

$\hspace{1.5cm}=f(g(t_{11}^{31}),\ldots,g(t_{1(i-1)}^{3(i-1)}),0,g(t_{1(i+1)}^{3(i+1)}),\ldots,g(t_{1m}^{3m}))\subseteq \delta(I)$.\\ 
as $I$ is a strongly $\phi$-$(2,2)$-absorbing $\delta$-primary hyperideal of $G$. Therefore, we get the result that 

$\hspace{1cm}f(g(t_{11}^{21}),\ldots,g(t_{1(i-1)}^{2(i-1)}),\phi(I),g(t_{1(i+1)}^{2(i+1)}),\ldots,g(t_{1m}^{2m}))$

$\hspace{1.5cm}=g(f(t_{11}^{1(i-1)},\phi(I),t_{1(i+1)}^{1m}),f(t_{21}^{2(i-1)},\phi(I),t_{2(i+1)}^{2m}))\in I/\phi(I)$\\ or  

$\hspace{1cm}f(g(t_{21}^{31}),\ldots,g(t_{2(i-1)}^{3(i-1)}),\phi(I),g(t_{2(i+1)}^{3(i+1)}),\ldots,g(t_{2m}^{3m}))$

$\hspace{1.5cm}=g(f(t_{21}^{2(i-1)},\phi(I),t_{2(i+1)}^{2m}),f(t_{31}^{3(i-1)},\phi(I),t_{3(i+1)}^{3m}))$

$\hspace{1.5cm}\in \delta(I)/\phi(I)=\delta_q(I/\phi(I))$\\
or 

$\hspace{1cm}f(g(t_{11}^{31}),\ldots,g(t_{1(i-1)}^{3(i-1)}),\phi(I),g(t_{1(i+1)}^{3(i+1)}),\ldots,g(t_{1m}^{3m}))$

$\hspace{1.5cm}=g(f(t_{11}^{1(i-1)},\phi(I),t_{1(i+1)}^{1m}),f(t_{31}^{3(i-1)},\phi(I),t_{3(i+1)}^{3m}))$

$ \hspace{1.5cm} \in \delta(I)/\phi(I)=\delta_q(I/\phi(I))$.\\
Consequently, $I/\phi(I)$ is a  weakly $(2,2)$-absorbing $\delta_q$-primary hyperideal of $G/\phi(I)$. 

$\Longleftarrow $ Assume that  $g(t_1^3) \in I-\phi(I)$ for  $t_1^3 \in G$. Then $f(g(t_1^3),\phi(I),0^{(m-2)}) \neq \phi(I)$. Therefore we conclude that

$\phi(I) \neq g(f(t_1,\phi(I),0^{(m-2)}),f(t_2,\phi(I),0^{(m-2)}),f(t_3,\phi(I),0^{(m-2)})) \in I/\phi(I)$. \\
 Since $I/\phi(I)$ is a  weakly $(2,2)$-absorbing $\delta_q$-primary hyperideal of $G/\phi(I)$, we have 
 
 $g(f(t_1,\phi(I),0^{(m-2)}),f(t_2,\phi(I),0^{(m-2)}))=f(g(t_1^2),\phi(I),0^{(m-2)}) \in I/\phi(I)$.\\
or 
 
 $g(f(t_2,\phi(I),0^{(m-2)}),f(t_3,\phi(I),0^{(m-2)}))  =f(g(t_2^3),\phi(I),0^{(m-2)}) \in \delta(I)/\phi(I)$.\\
or 
 
$g(f(t_1,\phi(I),0^{(m-2)}),f(t_3,\phi(I),0^{(m-2)})) =f(g(t_1^3),\phi(I),0^{(m-2)})  \in \delta(I)/\phi(I)$.\\
Hence we obtain  $g(t_1^2) \in I$ or $g(t_2^3) \in \delta(I)$ or $g(t_1^3) \in \delta(I)$. Thus, $I$  is   a $\phi$-$(2,2)$-absorbing primary hyperideal of $G$. 
\end{proof}
\begin{theorem} \label{madar}
Let $I$ be a proper hyperideal of $G$. Then $I$  is   a  $\phi$-$(k,n)$-absorbing $\delta$-primary hyperideal of $G$ if and only if   $I/\phi(I)$ is a   weakly $(k,n)$-absorbing $\delta_q$-primary hyperideal of $G/\phi(I)$. 
\end{theorem}
\begin{proof}
It can be easily proved in a similar manner to the proof of Theorem \ref{weak}.
\end{proof}
Let $I$ be a 
weakly $(k,n)$-absorbing primary 
 hyperideal of $G$. Recall from \cite{weak} that $(t_1^{kn-k+1})$ is called a $(k,n)$-zero primary element  of $I$ for $t_1^{kn-k+1} \in G$, if $g(t_1^{kn-k+1})=0$, $g(t_1^{(k-1)n-k+2}) \notin I$ and none of $g$-product of $(k-1)n-k+2$ of $t_i^,$s , other than $g(t_1^{(k-1)n-k+2})$, is in ${\bf r} (I)$. Now, suppose that $I$ is a weakly $(k,n)$-absorbing $\delta$-primary hyperideal of $G$. We say that $(t_1^{kn-k+1})$ is a $(k,n)$-zero $\delta$-primary element of $I$ for $t_1^{kn-k+1} \in G$, if $g(t_1^{kn-k+1})=0$, $g(t_1^{(k-1)n-k+2}) \notin I$ and none of $g$-product of $(k-1)n-k+2$ of $t_i^,$s , other than $g(t_1^{(k-1)n-k+2})$, is in $\delta(I)$. Moreover, let $I$ be a $\phi$-$(k,n)$-absorbing $\delta$-primary hyperideal of $G$. If  $g(t_1^{kn-k+1}) \in \phi(I)$, $g(t_1^{(k-1)n-k+2}) \notin I$ and none of $g$-product of $(k-1)n-k+2$ of $t_i^,$s , other than $g(t_1^{(k-1)n-k+2})$, is in $\delta(I)$, then $(t_1^{kn-k+1})$ is called a $\phi$-$(k,n)$-$\delta$-primary element of $I$.
\begin{theorem} \label{weak2}
Let $I$ be a  strongly $\phi$-$(2,2)$-absorbing $\delta$-primary hyperideal of $G$ and   $t_1^3\in G$. Then we have the following 
equivalent statements:
\begin{itemize}
\item[\rm(i)]~ $(t_1^3)$ is a $\phi$-$(2,2)$-$\delta$-primary element of $I$.
\item[\rm(ii)]~$(f(t_1,\phi(I),0^{(m-2)}),f(t_2,\phi(I),0^{(m-2)}),f(t_3,\phi(I),0^{(m-2)}))$ is a $(2,2)$-zero $\delta_q$-primary element of $I/\phi(I)$.
\end{itemize}
\end{theorem}
\begin{proof}
(i) $\Longrightarrow $ (ii) Let $(t_1^3)$ is a $\phi$-$(2,2)$-$\delta$-primary element of $I$. This implies that $g(t_1^3) \in \phi(I)$, $g(t_1^2) \notin I$, $g(t_2^3) \notin \delta(I)$ and $g(t_1,t_3) \notin \delta(I)$. It follows that  $f(g(t_1^2),I,0^{(m-2)}) \notin I/\phi(I)$ and $f(g(t_2^3),\phi(I),0^{(m-2)}) $, $f(g(t_1,t_3),\phi(I),0^{(m-2)}) \notin \delta(I)/\phi(I)$. Then we get  $(f(t_1,\phi(I),0^{(m-2)}),f(t_2,\phi(I),0^{(m-2)}),f(t_3,\phi(I),0^{(m-2)}))$ is a $(2,2)$-zero $\delta_q$-primary element of $I/\phi(I)$.

(i)  $\Longleftarrow $ (ii) Suppose that $(f(t_1,\phi(I),0^{(m-2)}),f(t_2,\phi(I),0^{(m-2)}),f(t_3,\phi(I),0^{(m-2)})$ is a $(2,2)$-zero $\delta_q$-primary of $I/\phi(I)$. Thus $g(t_1^3) \in \phi(I)$ but $f(g(t_1^2),I,0^{(m-2)}) \notin I/\phi(I)$, $f(g(t_2^3),\phi(I),0^{(m-2)}) \notin \delta(I)/\phi(I)$ and $f(g(t_1,t_3),\phi(I),0^{(m-2)}) \notin \delta(I)/\phi(I)$. Hence $g(t_1^3) \in \phi(I)$, $g(t_1^2) \notin I$, $g(t_2^3) \notin \delta(I)$ and $g(t_1,t_3) \notin \delta(I)$. It implies that $(t_1^3)$ is a $\phi$-$(2,2)$-$\delta$-primary of $I$.
\end{proof}
\begin{theorem} \label{kweak2}
 Let $I$ be a strongly $\phi$-$(k,n)$-absorbing $\delta$-primary hyperideal of $G$ and  $t_1^{kn-k+1} \in G$. Then the followings are equivalent:
\begin{itemize} 
\item[\rm(i)]~ $(t_1^{kn-k+1})$ is a $\phi$-$(k,n)$-$\delta$-primary of $I$.
\item[\rm(ii)]~$(f(t_1,\phi(I),0^{(m-2)}),\cdots,f(t_{kn-k+1},\phi(I),0^{(m-2)})$ is a $(k,n)$-zero $\delta$-primary of $I/\phi(I)$.
\end{itemize}
\end{theorem}
\begin{proof}
It is seen to be true in a similar manner to Theorem \ref{weak2}.
\end{proof}
The following lemma is needed in the proof of Theorem \ref{weak4}.
\begin{lem} \label{madar2}
Assume that $I$ is a strongly weakly $(k,n)$-absorbing $\delta$-primary hyperideal of $G$. If $(t_1^{kn-k+1})$ is a $(k,n)$-zero $\delta$-primary of $I$ for some $t_1^{kn-k+1} \in G$, then for every $i_1,\ldots,i_r \in \{1,\ldots,kn-k+1\}$ and $1 \leq r \leq (k-1)n-k+2$, \[g(t_1,\ldots,\widehat{t_{i_1}},\ldots,\widehat{t_{i_2}},\ldots,\widehat{t_{i_r}},\cdots,t_{kn-k+1},I^{(r)}) =0.\] 
\end{lem}
\begin{proof}
It can be easily seen that the claim is true in a similar manner to the proof of
Theorem 4.9 in \cite{weak}. 
\end{proof}

\begin{theorem} \label{weak4}
 Let $I$ be a strongly $\phi$-$(k,n)$-absorbing $\delta$-primary hyperideal of $G$ and $(t_1^{kn-k+1})$ be a $\phi$-$(k,n)$-$\delta$-primary of $I$ for some $t_1^{kn-k+1} \in G$. Then  for every $i_1,\ldots,i_r \in \{1,\ldots,kn-k+1\}$ and $1 \leq r \leq (k-1)n-k+2$,
 \[g(t_1,\ldots,\widehat{t_{i_1}},\ldots,\widehat{t_{i_2}},\ldots,\widehat{t_{i_r}},\ldots,t_{kn-k+1},I^{(r)}) \subseteq \phi(I).\]
\end{theorem}
\begin{proof}
Assume that $I$ is  a strongly $\phi$-$(k,n)$-absorbing $\delta$-primary hyperideal of $G$ and $(t_1^{kn-k+1})$ is a $\phi$-$(k,n)$-$\delta$-primary of $I$ for  $t_1^{kn-k+1} \in G$. Then we conclude that $I/\phi(I)$ is a strongly weakly $(k,n)$-absorbing $\delta_q$-primary hyperideal of $G/\phi(I)$ by Theorem \ref{madar} and   $(f(t_1,\phi(I),0^{(m-2)}),\ldots,f(t_{kn-k+1},\phi(I),0^{(m-2)}))$ is a $(k,n)$-zero $\delta_q$-primary of $I/\phi(I)$ by Theorem \ref{kweak2}.  Therefore for every $i_1,\ldots,i_r \in \{1,\ldots,kn-k+1\}$ and $1 \leq r \leq (k-1)n-k+2$, 

$f(g(f(t_1,\phi(I),0^{(m-2)}),\ldots,f(\widehat{t_{i_1},\phi(I),0^{(m-2)}}),\ldots,f(\widehat{t_{i_2},\phi(I),0^{(m-2)}}),\ldots,$

$\hspace{0.1cm}f(\widehat{t_{i_r},\phi(I),0^{(m-2)}}),\ldots,
 f(t_{kn-k+1},\phi(I),0^{(m-2)}),I^{(r)}),\phi(I),0^{(m-2)})= \phi(I)$\\
 by Lemma \ref{madar2}. Consequently, $g(t_1,\ldots,\widehat{t_{i_1}},\ldots,\widehat{t_{i_2}},\ldots,\widehat{t_{i_r}},\ldots,t_{kn-k+1},I^{(r)}) \subseteq \phi(I).$
\end{proof}
The following lemma is needed in the proof of our next result.
\begin{lemma} \label{tangsiri}
Assume that $I$ is a strongly weakly $(k,n)$-absorbing $\delta$-primary hyperideal of $G$. If $I$ has a $(k,n)$-zero $\delta$-primary element, then $g(I^{(kn-k+1)})=0$.
\end{lemma}
\begin{proof}
The proof can be  obtained in a similar manner to Theorem 4.10 in \cite{weak}.
\end{proof}
The following theorem shows that a strongly $\phi$-$(k,n)$-absorbing $\delta$-primary hyperideal $I$ of $G$  is  a $(k,n)$-absorbing $\delta$-primary hyperideal of $G$ when $g(I^{(kn-k-1)}) \nsubseteq \phi(I)$.
\begin{theorem} \label{madar3}
 If $I$ is a strongly $\phi$-$(k,n)$-absorbing $\delta$-primary hyperideal of $G$ such that $g(I^{(kn-k-1)}) \nsubseteq \phi(I)$, then $I$ is  a $(k,n)$-absorbing $\delta$-primary hyperideal of $G$.
\end{theorem}
\begin{proof}
Assume that $I$ is  not  a $(k,n)$-absorbing $\delta$-primary hyperideal of $G$. Then there exist $t_1^{kn-k+1} \in G$ such that $g(t_1^{kn-k+1}) \in \phi(I)$, $g(t_1^{(k-1)n-k+2}) \notin I$ and none of $g$-product of $(k-1)n-k+2$ of $t_i^,$s , other than $g(t_1^{(k-1)n-k+2})$, is in $\delta(I)$ which implies $(t_1^{kn-k+1})$ is  a $\phi$-$(kn-k+1)$-$\delta$-primary element of $I$. It follows that  $(f(t_1,\phi(I),0^{(m-2)}),\ldots,f(t_{kn-k+1},\phi(I),0^{(m-2)}))$ is a $(k,n)$-zero $\delta$-primary of $I/\phi(I)$ by Theorem \ref{kweak2}. Since $I$ is a strongly $\phi$-$(k,n)$-absorbing $\delta$-primary hyperideal of $G$, we conclude that $I/\phi(I)$ is a   strongly weakly $(k,n)$-absorbing $\delta_q$-primary hyperideal of $G/\phi(I)$ by Theorem \ref{madar}. Therefore  $f(g(I^{(kn-k+1)}),\phi(I),0^{(m-2)})=\phi(I)$ by Lemma \ref{tangsiri}, which means $g(I^{(kn-k-1)}) \subseteq \phi(I)$, a contradiction. Thus $I$ is  a $(k,n)$-absorbing $\delta$-primary hyperideal of $G$.
\end{proof}
In view of Theorem \ref{madar3}, we have the following result.
\begin{corollary} \label{weak5}
Let $I$ be a strongly $\phi$-$(k,n)$-absorbing $\delta$-primary hyperideal of $G$ but is not a $(k,n)$-absorbing $\delta$-primary. Then 
\begin{itemize} 
\item[\rm(i)]~$g(I^{(kn-k+1)}) \subseteq \phi(I)$.
\item[\rm(ii)]~${\bf r}(I)={\bf r}(\phi(I))$.
\end{itemize} 
\end{corollary}
\begin{proof}
(i) It follows from \ref{madar3}.

(ii) Assume that $I$ is a strongly $\phi$-$(k,n)$-absorbing $\delta$-primary hyperideal of $G$  that is not a $(k,n)$-absorbing $\delta$-primary hyperideal of $G$.  By (i), we have ${\bf r}(I) \subseteq {\bf r}(\phi(I))$. Moreover, since $\phi(I) \subseteq I$, we obtain ${\bf r}(\phi(I)) \subseteq {\bf r}(I)$. Then we get ${\bf r}(I)={\bf r}(\phi(I))$.
\end{proof}
Note that a proper hyperideal $I$ of $G$ may not be $\phi$-$(k,n)$-absorbing $\delta$-primary when it holds $g(I^{(kn-k+1)}) \subseteq \phi(I)$. The following example verifies the explanation. 
\begin{example} \label{mansooreh} 
Consider Krasner $(5,4)$-hyperring $(\mathbb{Z}_{7^{7s}},+,\cdot)$ for $s>6$   where $+$ and $\cdot$ are usual addition and multiplication. Define $\phi(Q)=Q^7$ for every $Q \in \mathcal{L}(\mathbb{Z}_{7^{7s}})$. In this hyperring, the hyperideal $I=\langle 7^s \rangle$ is not a $\phi$-$(2,4)$-absorbing $\delta_0$-primary hyperideal of $\mathbb{Z}_{7^{7s}}$ since $7\cdot7\cdot7\cdot7\cdot7\cdot7\cdot7^{s-6} \in I - \phi(I)$ but $7\cdot7\cdot7\cdot7, 7\cdot7\cdot7\cdot7^{s-6} \notin I$. However, $I^7 \subseteq \phi(I)$.
\end{example}

Suppose that $(G_1, f_1, g_1)$ and $(G_2, f_2, g_2)$ are two Krasner $(m, n)$-hyperrings. Recall from \cite{d1} that a mapping
$h : G_1 \longrightarrow G_2$ is  a homomorphism if for all $s^m _1, t^n_ 1 \in G_1$ we have
\begin{itemize} 
\item[\rm(i)]~ $ h(f_1(s_1^m)) = f_2(h(s_1),\ldots,h(s_m))$, 
\item[\rm(ii)]~$ h(g_1(t_1^n)) = g_2(h(t_1),\ldots,h(t_n)). $
\end{itemize} 
Suppose that $\phi$ and $\phi^{\prime}$ are reduction functions of hyperideals of $G_1$ and $G_2$, respectively. Furthermore, assume that $\delta$ and $\delta^{\prime}$ are  expansion  functions of hyperideals of $G_1$ and $G_2$, respectively.  Then the homomorphism  $h : G_1 \longrightarrow G_2$ is said to be a $(\delta,\phi)$-$(\delta^{\prime},\phi^{\prime})$-homomorphism if $\delta(h^{-1}(I_2))=h^{-1}(\delta^{\prime}(I_2))$ and $\phi(h^{-1}(I_2))=h^{-1}(\phi^{\prime}(I_2))$ for every hyperideal $I_2$ of $G_2$. For instance, if $\phi$ and $\phi^{\prime}$ are  empty reduction functions of hyperideals of $G_1$ and $G_2$, respectively, and also $\delta$ and $\delta^{\prime}$ are radical operations on hyperideals of $G_1$ and $G_2$, respectively, then every homorphism $h : G_1 \longrightarrow G_2$ is a $(\delta,\phi)$-$(\delta^{\prime},\phi^{\prime})$-homomorphism.
Let   $I_1$ be a hyperideal of $G_1$ with $Ker (h) \subseteq I_1$ and let $h : G_1 \longrightarrow G_2$ be a $(\delta,\phi)$-$(\delta^{\prime},\phi^{\prime})$-epimorphism. It is easy to see that $h(\delta(I_1))=\delta^{\prime}(h(I_1))$ and  $h(\phi(I_1))=\phi^{\prime}(h(I_1))$.
\begin{theorem} \label{homo}
Let $\delta$ and $\delta^{\prime}$ be expansion functions of hyperideals $G_1$ and $G_2$, respectively, and let
$\phi$ and $\phi^{\prime}$ be reduction functions of hyperideals $G_1$ and $G_2$, respectively, such that $h:G_1 \longrightarrow G_2$  is  a $(\delta,\phi)$-$(\delta^{\prime},\phi^{\prime})$-homomorphism.
\begin{itemize} 
\item[\rm(1)]~ If $I_2$ is a $\phi^{\prime}$-$(k,n)$-absorbing $\delta^{\prime}$-primary hyperideal of $G_2$, then $h^{-1}(I_2)$ is a $\phi$-$(k,n)$-absorbing $\delta$-primary hyperideal of $G_1$. 
\item[\rm(2)]~If $h$ is surjective and $I_1$ is a $\phi$-$(k,n)$-absorbing $\delta$-primary hyperideal of $G_1$ with $Ker (h) \subseteq I_1$, then $h(I_1)$ is a $\phi^{\prime}$-$(k,n)$-absorbing $\delta^{\prime}$-primary hyperideal of $G_2$.
\end{itemize} 
\end{theorem}
\begin{proof}
(1) Assume that $I_2$ is a $\phi^{\prime}$-$(k,n)$-absorbing $\delta^{\prime}$-primary hyperideal of $G_2$. Let $g(t_1^{kn-k+1}) \in h^{-1}(I_2)-\phi(h^{-1}(I_2))$ for $t_1^{kn-k+1} \in G_1$  . Therefore we obtain $h(g(t_1^{kn-k+1}))=g(h(t_1),\ldots,h(t_{kn-k+1})) \in I_2-\phi^{\prime}(I_2)$. By the hypothesis, we conclude that either $g(h(t_1),\ldots,h(t_{(k-1)n-k+2}))=h(g(t_1^{(k-1)n-k+2})) \in I_2$ or 
 $g(h(t_1),\ldots,\widehat{h(t_i)},\ldots,h(t_{kn-k+1)})=h(g(t_1,\ldots,\widehat{t_i},\ldots,t_{kn-k+1)})) \in \delta^{\prime}(I_2)$. Then we obtain  $g(t_1^{(k-1)n-k+2}) \in h^{-1}(I_2)$, 
 or $g(t_1,\ldots,\widehat{t_i},\ldots,t_{kn-k+1)}) \in h^{-1}(\delta^{\prime}(I_2))=\delta(h^{-1}(I_2))$ for some $1 \leq i \leq n$. Thus $h^{-1}(I_2)$ is a $\phi$-$(k,n)$-absorbing $\delta$-primary hyperideal of $G_1$.

 (2) Suppose that $I_1$ is a $\phi$-$(k,n)$-absorbing $\delta$-primary hyperideal of $G_1$ with $Ker (h) \subseteq I_1$.  Choose $s_1^{kn-k+1} \in G_2$ such that $g(s_1^{kn-k+1}) \in h(I_1)-\phi^{\prime}(h(I_1))$. Then  there exist $t_1^{kn-k+1} \in G_1$ with $h(t_1)=s_1, \ldots, h(t_{kn-k+1})=s_{kn-k+1}$. Therefore we have  $h(g(t_1^{kn-k+1})=g(h(t_1),\ldots,h(t_{kn-k+1}))=g(s_1^{kn-k+1}) \in h(I_1)$. Since  $Ker (h) \subseteq I_1$ and  $h$ is surjective,  we obtain $g(t_1^{kn-k+1}) \in I_1-\phi(I_1)$. By the hypothesis,  we have $g(t_1^{(k-1)n-k+2)}) \in I_1$ or $g(r_1,\ldots,\widehat{r_i},\ldots,r_{kn-k+1}) \in \delta(I_1)$ for some $1 \leq i \leq (k-1)n-k+2$. Therefore we conclude that
 
$ \hspace{2cm}h(g(t_1^{(k-1)n-k+2})=g(h(t_1),\ldots,h(t_{(k-1)n-k+2})) $

$\hspace{4.8cm}=g(s_1^{(k-1)n-k+2}) \in h(I_1)$\\ or 

 $\hspace{0.4cm}h(g(t_1,\ldots,\widehat{t_i},\ldots,t_{kn-k+1})=g(h(t_1),\ldots,\widehat{h(t_i)},\ldots,h(t_{kn-k+1}))$
 
 $\hspace{4.8cm}=g(s_1,\ldots,\widehat{s_i},\ldots,s_{kn-k+1})$
 
 $\hspace{4.8cm} \in h(\delta(I_1))$
 
 $\hspace{4.8cm}=\delta^{\prime}(h(I_1)).$\\  Thus  $h(I_1)$ is a $\phi^{\prime}$-$(k,n)$-absorbing $\delta^{\prime}$-primary hyperideal of $G_2$.
\end{proof}
Let $J$ is a hyperideal of $G$. The surjective map $\pi$ from $G$ to $G/J$, defined by $\pi(t)=f(t,J,0^{(m-2)})$ is homomorphism by Theorem 3.2 in \cite{sorc1}.
Assume that  $I$ is  a hyperideal of $G$ and $\phi$ is a reduction function of $\mathcal{L}(G)$. Then the function $\phi_q$ from $\mathcal{L}(G/Q)$ to $\mathcal{L}(G/Q) \cup \{\varnothing\}$ defined by $\phi_q(I/Q)=\phi(I)/Q$ is a reduction function. Now, let us give the following theorem as a result of Theorem \ref{homo}.
\begin{theorem}
Let $I$ and $J$ be  two  hyperideals of $G$ and $\phi$ be a reduction function of hyperideals of $G$ such that $J \subseteq \phi(I) \subseteq I$. If $I$ is a $\phi$-$(k,n)$-absorbing $\delta$-primary hyperideal of $G$, then $I/J$ is a $\phi_q$-$(k,n)$-absorbing $\delta_q$-primary hyperideal of $G/J$.
\end{theorem}




 
 
Assume that 
 $(G_1, f_1, g_1)$ and $(G_2, f_2, g_2)$ are two Krasner $(m,n)$-hyperrings such that $1_{G_1}$ and $1_{G_2}$ are scalar identitis of $G_1$ and $G_2$, respectively. Then 
  $(G_1 \times G_2, f_1\times f_2 ,g_1 \times g_2 )$ is a Krasner $(m, n)$-hyperring where

$f=f_1 \times f_2((p_{1}, q_{1}),\ldots,(p_m,q_m)) = \{(p,v) \ \vert \ \ q \in f_1(p_1^m), q \in f_2(q_1^m) \}$

$g=g_1 \times g_2 ((s_1,t_1),\ldots,(s_n,t_n)) =(g_1(s_1^n),g_2(t_1^n)) $,\\
for all $p_1^m,s_1^n \in G_1$ and $q_1^m,t_1^n \in G_2$ \cite{car}. Suppose that $\delta_i$ and $\phi_i$ are an expansion  function  
 and a reduction function of hyperideals of $G_i$, respectively, for each $i \in \{1,2\}$. If $I_i$ is a hyperideal of $G_i$ for $i \in \{1,2\}$, then $\bar{\delta}$ and $\bar{\phi}$ defined as $\bar{\delta}(I_1 \times I_2)=\delta_1(I_1) \times \delta_2(I_2)$ and $\bar{\phi}(I_1 \times I_2)=\phi_1(I_1) \times \phi_2(I_2)$ are reduction and expansion functions of hyperideals $G_1 \times G_2$, respectively.

 \begin{theorem} \label{car}
 Let $I_i$ be a proper hyperideal of    a commutative
Krasner $(m,n)$-hyperring $(G_i,f_i,g_i)$ for each $i \in \{1,\ldots,kn-k+1\}$, $\delta_i $ an expansion function of hyperideal of $G_i$,  $\phi_i$  a reduction function of hyperideal  of $G_i$.   If $I_1\times \cdots \times I_{kn-k+1}$ is $\bar{\phi}$-$(k+1,n)$-absorbing $\bar{\delta}$-primary hyperideal of $G_1\times \cdots \times G_{kn-k+1}$, then $I_i$ is a   $\phi_i$-$(k,n)$-absorbing $\delta_i$-primary hyperideal of $G_i$ for all $i \in \{1,\ldots,kn-k+1\}$.
 \end{theorem}
\begin{proof}
Put $G=G_1\times \cdots \times G_{kn-k+1}$ and $I=I_1\times \cdots \times I_{kn-k+1}$. Let $t_1^{kn-k+1} \in G_i$ such that $g_i(t_1^{kn-k+1}) \in I_i-\phi_i(I_i)$. Suppose by contradiction that $I_i$ is not a $\phi_i$-$(k,n)$-absorbing $\delta$-primary hyperideal of $G_i$ for some $i \in \{1,\ldots,kn-k+1\}$. Let

$r_1=(1_{G_1},\ldots, 1_{G_{i-1}},t_1,1_{G_{i+1}},\ldots, 1_{G_{kn-k+1}}),$

$ r_2=(1_{G_1},\ldots, 1_{G_{i-1}},t_2,1_{G_{i+1}},\ldots, 1_{G_{kn-k+1}}),$

$\vdots$

$r_{kn-k+1}=(1_{G_1},\ldots, 1_{G_{i-1}},t_{kn-k+1},1_{G_{i+1}},\ldots,1_{G_{kn-k+1}}), $

$r_{kn-k+2}=(1_{G_1},\ldots, 1_{G_{i-1}},1_{G_i},1_{G_{i+1}},\ldots,1_{G_{kn-k+1}}),$

$ r_{(k+1)n-(k+1)+1}=(0,\ldots, 0,1_{G_i},0,\ldots,0)$.\\
Hence $g(r_1^{(k+1)n-(k+1)+1)}) \in I-\phi(I)$ but $g(r_1^{kn-k+1}) \notin I$. Since $I$ is a $\phi$-$(k+1,n)$-absorbing hyperideal of $G$,  we conclude that one of $g$-productions of $kn-k+1$ of $r_i^,$s containing the element $ r_{(k+1)n-(k+1)+1}$ is in $\bar{\delta}(I)$. This implies that there exist $(k-1)n-k+2$ of $t_i^,$s whose $g_i$-product is in $\delta_i(I_i)$, a contradiction. Consequently,  $I_i$ is a $\phi_i$-$(k,n)$-absorbing $\delta_i$-primary hyperideal of $G_i$.
\end{proof}

\section{Conclusion}

In this paper, we introduced and studied the concept of $\phi$-$(k,n)$-absorbing $\delta$-primary hyperideals presenting a unified framework containing several classical hyperideals types. For instance, weakly $(k,n)$-absorbing hyperideals arise when   $\phi$ and $\delta$ are zero and identity functions, respectively. We delved deeply into the theory of $\phi$-$(k,n)$-absorbing $\delta$-primary hyperideals, investigating their fundamental properties, characterizations and connections to other types of hyperideals. We examined whether the union and intersection of the
collection of $\phi$-$(k,n)$-absorbing $\delta$-primary hyperideals preserve the algebraic structure.
One of the notable results we established is that if $I$ is a strongly $\phi$-$(k,n)$-absorbing $\delta$-primary hyperideal of $G$ that is not a $(k,n)$-absorbing $\delta$-primary, then $g(I^{(kn-k+1)}) \subseteq \phi(I)$. 
Moreover, we examined the attitude of  the $\phi$-$(k,n)$-absorbing $\delta$-primary hyperideals  under homomorphisms, in quotient hyperring,  and in cartesian products.


\end{document}